\documentclass[preprint,10pt]{todarticle}

\usepackage{amssymb,amsfonts,amsmath,amscd,amsthm,latexsym}
\usepackage{graphicx}
\usepackage{mathptmx}      
\usepackage[english]{babel}
\usepackage{geometry}
\usepackage[all]{xy}
\usepackage{appendix}
\usepackage[active]{srcltx}
\usepackage{color}

\newtheorem{theo}{Theorem}[section]

\newtheorem{lemm}[theo]{Lemma}
\newtheorem{prop}[theo]{Proposition}
\newtheorem{subsec}[theo]{Situation}
\theoremstyle{definition}
\newtheorem{defi}[theo]{Definition}

\newtheorem{rema}[theo]{Remark}

\newcommand\numberthis{\addtocounter{equation}{1}\tag{\theequation}}
\newcommand{\bbZ}{\mathbb{Z}}
\newcommand{\Ker}{\mathrm{Ker}}
\newcommand{\Ima}{\mathrm{Im}}
\newcommand{\id}{\mathrm{id}}
\newcommand{\Hc}{\mathrm{H}}
\newcommand{\HH}{\mathrm{HH}}
\newcommand{\Hom}{\mathrm{Hom}}

\begin{document}

\begin{frontmatter}



\title{BD algebras and group cohomology}

\author[UT]{Constantin-Cosmin Todea \corref{cor1}}
\ead{Constantin.Todea@math.utcluj.ro}
\cortext[cor1]{Corresponding author}
\address[UT]{Department of Mathematics, Technical University of Cluj-Napoca, Str. G. Baritiu 25,  Cluj-Napoca 400027, Romania}
\fntext[label2]{This work was supported by a grant of the Ministry of Research, Innovation and Digitization, CNCS/CCCDI–-UEFISCDI, project number PN-III-P1-1.1-TE-2019-0136, within PNCDI III}
\begin{abstract}  BD algebras (Beilinson-Drinfeld algebras) are algebraic structures which are defined similarly to  BV algebras (Batalin-Vilkovisky algebras). The equation defining the BD operator has the same structure
as the equation for BV algebras, but the BD operator is increasing with degree +1. We obtain methods of constructing  BD algebras  in the context of group cohomology.
\end{abstract}


\end{frontmatter}

\section{Introduction}\label{sec1}
For any associative algebra  the associated Hochschild cohomology group  has a rich structure. It is a graded commutative algebra via the cup product and, 
it has a graded Lie bracket of degree $-1$ obtaining what is now called Gerstenhaber algebra \cite{Ge}. A new structure in Hochschild theory, BV algebra,  has been extensively studied in topology and mathematical physics for a long time, see \cite{Get} to name just one reference.  More recently, this so-called
Batalin-Vilkovisky structure was introduced into algebra. Roughly speaking a Batalin-Vilkovisky structure is an
operator on Hochschild cohomology which squares to zero and which, together with the cup product,
can express the Gerstenhaber bracket. A BV structure only exists over Hochschild cohomology of certain special
classes of algebras. In \cite{Tr} the author defines a BV structure on the Hochschild cohomology of symmetric algebras. We mention some references where BV algebras specialized for Hochschild cohomology of group algebras are studied: \cite{AnDu}, in which the authors determine the BV algebra structure of the Hochschild cohomology of group algebras of cyclic groups of prime order; \cite{BeKeLi,LiZh}, which are dedicated to  explicit descriptions of the BV operators for Hochschild cohomology of group algebras and computations for some particular finite groups;  etc.

The structures called BD algebras (Beilinson-Drinfeld algebras), which appears in mathematical physics,  have a superficial similarity with BV algebras and are both related to BV-BRST formalism in physics, see \cite[Definition A.3.5]{CoGw}, the following remarks in \cite[p. 289]{CoGw} and the discussion in \cite{BVBD}.  The equation defining the BD operator $\Delta_{BD}$ has the same structure as the equation for  BV algebras if we replace the BV operator $\Delta_{BV}$ by $\Delta_{BD}$  (and remove some factor  denoted $\overline{h}$ in \cite{CoGw}). But since  $\Delta_{BD}$ is of degree $+1$ the similarity is only superficial, as mentioned in \cite{BVBD}. Section \ref{sec3} is devoted to recall these structures under the unified framework of Poisson $r$-algebras (where $r$ is a nonnegative integer) and $r$-Batalin-Vilkovisky algebras, see \cite[Section 1]{CaFiLo}, but with the useful change of indexing due to Costello and Gwilliam \cite[Section A.3.2]{CoGw}.

The main goal of this article is to introduce BD algebra structure from mathematical physics into algebra. Although most of the results of this paper work over any group we prefer finite groups. Let $G$ be a finite group, $k$ be a commutative ring,  $K$ be a field and $p$ be any prime. The article is organized as follows. In Section \ref{sec2} we generalize the fact that group cohomology ring $\Hc^*(G, \bbZ/p\bbZ)$ together with a connecting homomorphism  (called Bockstein map), which is induced by a given short exact sequence of trivial $\bbZ G$-modules, is a DG-algebra, see \cite[Lemma 4.3.2 and 4.3.3]{BeII}. Let $A$ be a $kG$-module which is also a $k$-algebra on which $G$ acts as automorphisms. The main result of this section is Theorem \ref{thm1}, in which  we show that under some conditions, collected in Situation \ref{sit*}, there are connecting homomorphisms $\theta_{G,A}^*:
\Hc^*(G,A)\rightarrow\Hc^{*+1}(G,A)$, see Remark \ref{rem22},  such that $(\Hc^*(G,A),\theta_{G,A}^*)$ becomes a DG-algebra. In Section \ref{sec3}, as mentioned above, we recall the definition of  BD algebras compared to that of  BV algebras. With the help  of some ideas of Section \ref{sec2}, Section \ref{sec4} is devoted to collect in   Situation \ref{sit**} conditions and assumptions which allow us to define a family of $k$-linear maps $\{\theta_{G,A,x}^*\}_{x\in G}$ such that  $(\Hc^*(G,A),\theta_{G,A,x}^*)$ is a DG-algebra for any $x\in G$. Moreover $\theta_{G,A,xy}^*=\theta_{G,A,x}^*+\theta_{G,A,y}^*$ for any $x,y\in G$, see Proposition \ref{prop42}. Starting with Definition \ref{defBD43} we work with  a $KG$-module $A$ which is also a $K$-algebra on which $G$ acts as automorphisms. Under the assumptions which are mentioned in Definition \ref{defBD43} we define a multiplication "$\cdot$", a bilinear map "$[\cdot,\cdot]$" and an operator $\Delta_{BD}$ such that in the second main result of this paper, Theorem \ref{thm2}, we show that if $G$ is abelian then $(KG\otimes \Hc^*(G,A),\cdot, [\cdot,\cdot],\Delta_{BD})$ is a BD algebra.
In last section we give short exact sequences of  $\bbZ C_3$-modules which are in Situation \ref{sit**} and for which Theorem \ref{thm2} can be applied. We will show in Subsection \ref{subsec53} that there is a  BD algebra structure on $\mathbb{F}_3C_3\otimes \Hc^*(C_3,\mathbb{F}_3)$. By \cite{CiSo} there is a canonical identification $\mathbb{F}_3C_3\otimes \Hc^*(C_3,\mathbb{F}_3)\cong \HH^*(\mathbb{F}_3C_3)$. It follows that we can transport, through this identification, the BD algebra structure on the Hochschild cohomology $\HH^*(\mathbb{F}_3C_3)$ of the cyclic $3$ group algebra.  If we can develop a theory for BD operators on Hochschild cohomology of various group algebras (of other classes of finite dimensional algebras) is an objective which remains open for future research.

We adopt the following notations and conventions. If $(C^*,d)$ is a cochain complex we call $d$ a differential and we say that $d:C^*\rightarrow C^{*+1}$ is of degree $+1$. Shifting by $r$ means $C^*[r]:=C^{*+r}$ and we call $d'$ a map of degree $r$ if $d':C^*\rightarrow C^{*+r}$. When $\alpha$ is a homogeneous element of $C^*$ such that $\alpha\in C^n$ we sometimes denote its degree by $|\alpha|=n$. If $S$ is a subset of $A$ and $f:A\rightarrow B$ is a map, we denote by $\id_A:A\rightarrow A$ the identity map on $A$ and
by $f|_{S}:S\rightarrow B$ the restriction of $f$ to $S.$
We shall use the notation $\Hc^*(G,A)$  with two meanings:  on one hand we have $\Hc^*(G,A)=\bigoplus_{n\geq 0}\Hc^n(G,A)$ as a $\bbZ$-graded $k$-module, on the other hand we could refer
to $\Hc^*(G,A)$ as the group cohomology in degree $*$. The reader should understand the difference in the respective context.

\section{Group cohomology with nontrivial coefficients as DG-algebras}\label{sec2}
Let $m,n$ be nonnegative integers and let $A,B,C$ be three  $kG$-modules.

\subsection{Connecting homomorphisms in group cohomology,  explicitly}\label{sec21}
We consider the coboundary differential
$$\partial_A^n:C^n(G,A)\rightarrow C^{n+1}(G,A),$$ $$\partial_A^n(\varphi)(g_1,\ldots,g_{n+1})=
g_1\varphi(g_2,\ldots,g_{n+1})+\sum_{i=1}^n(-1)^i\varphi(g_1,\ldots,g_ig_{i+1},\ldots,g_{n+1})+(-1)^{n+1}\varphi(g_1,...,g_n),$$
for any $g_1,\ldots,g_{n+1}\in G$, $\varphi\in C^n(G,A)$; where $$C^n(G,A)=\{\varphi|\varphi:G^{\times n}\rightarrow A \ \text{set map}\}.$$
Group cohomology of $G$ with coefficients in $A$ is 
$$\Hc^n(G,A):=\Ker \partial_A^n / \Ima \partial_A^{n-1}\overset{not}{=}Z^n(G,A)/B^n(G,A),\quad n\in\bbZ, n\geq 1.$$
If $n=0$ then $C^0(G,A)$ is identified with $A$ and $$\Hc^0(G,A)=\Ker\partial_A^0=A^G.$$

Let \begin{equation}\label{eq1}
\xymatrix{0\ar[r]&A\ar[r]^-{\iota}
&B\ar[r]^-{\pi}&C\ar[r] &0}
\end{equation}  be a short exact sequence of $kG$-modules.
Since $\iota$ is injective there is an isomorphism of $kG$-modules 
$$\iota':A\rightarrow\Ker\pi,\iota'(a)=\iota(a),$$
for any $a\in A.$
Let $r:\Ker\pi\rightarrow A$ be the inverse of $\iota'$. 
Since $\pi$ is surjective, it has a section (right inverse set map) which we denote by $s:C\rightarrow B$, hence
$\pi\circ s=\id_C$.
The following lemma is a well-known ingredient used to define explicitly the connecting homomorphism induced in group cohomology by the Long Exact Sequence Theorem.
\begin{lemm}\label{lem22}
With the notations above we have
\begin{itemize}
    \item[(a)] If $\varphi\in Z^n(G,C)$ then $\Ima\partial_B^n(s\circ\varphi)\subseteq\Ker\pi$;
    \item[(b)] If $\varphi\in Z^n(G,C)$ then $r\circ \partial_B^n(s\circ\varphi)\in Z^{n+1}(G,A)$.
\end{itemize}
\end{lemm}
By the Long Exact Sequence Theorem applied to (\ref{eq1}) there is a "connecting homomorphism", which we denote by
$$\theta_{G,C,A}^n:\Hc^n(G,C)\rightarrow \Hc^{n+1}(G,A),$$
$$\overline{\varphi}\mapsto\overline{\theta_{G,C,A}^n(\varphi)},\varphi\in Z^n(G,C).$$
Lemma \ref{lem22} gives us the opportunity to define $\theta_{G,C,A}^*$ explicitly as follows
$$\theta_{G,C,A}^n(\varphi)=r\circ\partial_B^n(s\circ\varphi),$$
for any $\varphi\in Z^n(G,C)$. 

\begin{rema}\label{rem22}
 If $A=C$ for shortness we denote $\theta_{G,C,A}^*$ by $\theta_{G,A}^*$.
\end{rema}

\subsection{Reminder on cup-products in group cohomology}\label{subsec22}
The cup product is
    $$\cup:\Hc^m(G,A)\otimes \Hc^n(G,B)\rightarrow \Hc^{m+n}(G,A\otimes B),$$
    $$(\overline{\varphi}\in \Hc^m(G,A), \overline{\psi}\in \Hc^n(G,B))\mapsto\overline{\varphi\cup\psi}\in \Hc^{m+n}(G,A\otimes B),$$
    $$(\varphi\cup\psi)(g_1,\ldots,g_m,g_{m+1},\ldots,g_{m+n}):=\varphi(g_1,\ldots,g_m)\otimes g_1 \ldots g_m \psi (g_{m+1},\ldots, g_{m+n}),$$
    for any $\varphi\in Z^m(G,A),\psi\in Z^n(G,B)$ and any $g_1,\ldots,g_{m+n}\in G.$
    
 If  $\mu:A\otimes B\rightarrow C$ is a $kG$-homomorphism we obtain the \textit{cup product with respect to the pairing} $\mu$
    $$\cup:\Hc^m(G,A)\otimes \Hc^n(G,B)\rightarrow \Hc^{m+n}(G,C).$$
    
 If $C=B=A$ is a $k$-algebra on which $G$ acts as automorphisms and $\mu $ is the structure multiplicative map (i.e. $\mu:A\otimes A\rightarrow A\ \mu(a\otimes b)= ab$) then there is
    $$\Hc^m(G,A)\otimes \Hc^n(G,A)\overset{\cup}{\rightarrow } \Hc^{m+n}(G,A\otimes A)\overset{\mu}{\rightarrow}\Hc^{m+n}(G,A),$$
     $$(\overline{\varphi}\in \Hc^m(G,A), \overline{\psi}\in \Hc^n(G,A))\mapsto\overline{\varphi\cup\psi}\in \Hc^{m+n}(G,A),$$
    $$(\varphi\cup\psi)(g_1,\ldots,g_{m+n}):=\varphi(g_1,\ldots,g_m)(g_1 \ldots g_m \psi (g_{m+1},\ldots ,g_{m+n})),$$
    for any $\varphi\in Z^m(G,A),\psi\in Z^n(G,B)$ and $g_1,\ldots,g_{m+n}\in G$.
    Thus $(\underset{n\geq 0}{\bigoplus}\Hc^n(G,A),\cup)$ becomes a graded $k$-algebra.

\subsection{ $\theta_{G,A}^*$ as graded derivation}
We need to assume a consistent number of conditions which we collect in the following lines.
Clearly if $A$ is a $k$-algebra on which $G$ is acting $k$-linearly then $A$ inherits a structure of left $kG$-module
induced by this action.

\begin{subsec}\label{sit*}
Let $A,B$ be $k$-algebras with $G$ acting by automorphisms on $A$
such that there is a short exact sequence of $kG$-modules $\xymatrix{0\ar[r]&A\ar[r]^-{\iota}
&B\ar[r]^-{\pi}&A\ar[r] &0}$ 
for which:
\begin{itemize}
    \item $\Ker \pi$ is an ideal in $B$;
    \item there exist two maps  $r:\Ker\pi\rightarrow A$, the inverse of the isomorphism $\iota':A\rightarrow \Ker\pi$,
    and $s:A\rightarrow B$ a section of $\pi$, satisfying:
   
          \begin{equation}\label{eq2setup*} r(b_1(g b_2)) = r(b_1)(g\pi(b_2)), \end{equation}
        \begin{equation}\label{eq3setup*} r(b_1'(g b_2')) = \pi(b_1')(g r(b_2')),\end{equation}
         \begin{equation}\label{eq4setup*} s(a_1(g a_2)) - s(a_1)(g s(a_2))\in\Ker\pi, \end{equation}
 for any  $b_1,b_2'\in\Ker\pi, b_1',b_2\in B,  a_1\in A, a_2\in A,\ g\in G.$
\end{itemize}
\end{subsec}
\begin{rema}\label{rem24} With the above notations, if $\pi$ is a homomorphism of $k$-algebras then Situation \ref{sit*} is shorter: relation (\ref{eq4setup*}) is automatically satisfied and $\Ker \pi$ is obviously an ideal in $B$. 
\end{rema}
\begin{lemm}\label{lem3} Assume we are in Situation \ref{sit*}.
\begin{itemize}
    \item[(a)] If $\alpha\in C^m(G,\Ker\pi), \beta\in C^n(G,B)$ then:
    \begin{itemize}
        \item[(i)] $\alpha\cup\beta\in C^{m+n}(G,\Ker\pi);$
        \item[(ii)] $r\circ(\alpha\cup\beta) = (r\circ\alpha)\cup(\pi\circ\beta).$
    \end{itemize}
    \item[(b)] If $\alpha\in C^m(G,B),\beta\in C^n(G,\Ker\pi)$ then:
    \begin{itemize}
       \item[(i)] $\alpha\cup\beta\in C^{m+n}(G,\Ker\pi)$;
       \item[(ii)] $r\circ(\alpha\cup\beta) = (\pi\circ\alpha)\cup(r\circ\beta).$
    \end{itemize}
    \item[(c)] If $\alpha\in C^m(G,A),\beta\in C^n(G,A)$ then for any $g_1,\ldots,g_{m+n}\in G$ we have
    $$(s\circ(\alpha\cup\beta)-(s\circ\alpha)\cup(s\circ\beta))(g_1,\ldots,g_m,\ldots,g_{m+n})\in\Ker\pi.$$
\end{itemize}
\end{lemm}
\begin{proof}
  Let $g_1,\ldots,g_m,\ldots,g_{m+n}\in G$. 
  \begin{itemize}
      \item[(a)]
For statement (i)  we have $$(\alpha\cup\beta)(g_1,\ldots,g_{m+n})\overset{\ref{subsec22}}{=}\alpha(g_1,\ldots,g_m)(g_1\ldots g_m\beta(g_{m+1},\ldots,g_{m+n}))$$ which is in $\Ker\pi$.

 For statement (ii) we obtain
 \begin{align*}(r\circ(\alpha\cup\beta))(g_1,\ldots,g_{m+n})&=r((\alpha\cup\beta)(g_1,\ldots,g_m,\ldots,g_{m+n})) \\&= r(\alpha(g_1,\ldots,g_m)(g_1\ldots g_m\beta(g_{m+1},\ldots,g_{m+n})))
          \\&\overset{(\ref{eq2setup*})}{=}(r\circ\alpha)(g_1,\ldots,g_m)(g_1\ldots g_m(\pi\circ\beta)(g_{m+1},\ldots,g_{m+n}))
          \\&=((r\circ\alpha)\cup(\pi\circ\beta))(g_1,\ldots,g_{m+n}).
   \end{align*}
   \item[(b)] The  proof is analogous to the proof of statement (a). 
      \item[(c)] 
      
 \begin{align*}&\pi((s\circ(\alpha\cup\beta)-(s\circ\alpha)\cup(s\circ\beta))(g_1,\ldots,g_m,\ldots,g_{m+n}))
      \\&=\pi(s(\alpha(g_1,\ldots, g_m)(g_1\ldots g_m\beta(g_{m+1},\ldots,g_{m+n})))-s(\alpha)(g_1,\ldots,g_{m}))(g_1\ldots g_{m}(s(\alpha)(g_{m+1},\ldots,g_{m+n})))
      \\ &=0,
      \end{align*}
      where the second equality holds by the definition of cup product and the last last equality is true by (\ref{eq4setup*}) of Situation \ref{sit*}.
  \end{itemize}
  \end{proof}
\begin{prop}\label{prop4thetagraded}
Assume Situation \ref{sit*} is satisfied. Then $\theta_{G,A}^*:\Hc^*(G,A)\rightarrow \Hc^{*+1}(G,A)$ is a graded derivation.
\end{prop}

\begin{proof}
Let $\overline{\alpha}\in \Hc^m(G,A),\overline{\beta}\in \Hc^n(G,A),\alpha\in Z^m(G,A),\beta\in Z^n(G,A),$ and  $g_1,\ldots,g_{m+n+1}\in G $.\\
By Lemma \ref{lem3}, (c) there is a map $\psi\in C^{m+n}(G,\Ker\pi)$ 
such that $$s\circ(\alpha\cup\beta) = (s\circ\alpha)\cup(s\circ\beta)+\psi.$$
The following assertions hold by applying heavily Lemma \ref{lem3} (a), (b) and the well-known fact that $\partial_B$ is a graded derivation
\begin{align*}
\theta_{G,A}^{m+n}&(\alpha\cup\beta)(g_1,\ldots,g_{m+n+1})
\\&=(r\circ\partial_B^{m+n}(s\circ(\alpha\cup\beta)))(g_1,\ldots,g_{m+n+1})
\\&=r(\partial_B^{m+n}(s\circ(\alpha\cup\beta)))(g_1,\ldots,g_{m+n+1})
\\&=(r\circ\partial_B^{m+n}((s\circ\alpha)\cup(s\circ\beta)))(g_1,\ldots,g_{m+n+1})+(r\circ\partial_B^{m+n}(\psi))(g_1,\ldots,g_{m+n+1})
\\&=(r\circ(\partial_B^m(s\circ\alpha)\cup s\circ\beta+(-1)^m(s\circ\alpha)\cup\partial_B^n(s\circ\beta)))(g_1,\ldots,g_{m+n+1})
\\&+(r\circ\partial_B^{m+n}(\psi))(g_1,\ldots,g_{m+n+1})
\\&=(r\circ(\partial_B^m(s\circ\alpha)\cup (\pi\circ s\circ\beta)+(-1)^m(\pi\circ s\circ\alpha)\cup(r\circ \partial_B^n(s\circ\beta))))(g_1,\ldots,g_{m+n+1})
\\&+(r\circ\partial_B^{m+n}(\psi))(g_1,\ldots,g_{m+n+1})
\\&=(\theta_{G,A}^m(\alpha)\cup\beta +(-1)^m\alpha\cup\theta_{G,A}^n(\beta)+r\circ\partial_B^{m+n}(\psi))(g_1,\ldots,g_{m+n+1})
\end{align*}
Next, we obtain
\begin{align*}
(r\circ&\partial_B^{m+n}(\psi))(g_1,\ldots,g_{m+n+1})
\\&=r(g_1\psi(g_2,\ldots,g_{m+n+1})+\ldots+(-1)^{m+n+2}\psi(g_1,\ldots,g_{m+n+1}))
\\&=g_1(r\circ\psi)(g_2,\ldots,g_{m+n+1})+\ldots+(-1)^{m+n+1}(r\circ\psi)(g_1,\ldots,g_{m+n+1})
\\&=\partial_A^{m+n}(r\circ\psi)(g_1,\ldots,g_{m+n+1}),
\end{align*}
hence
$r\circ\partial_B^{m+n}(\psi)\in B^{m+n}(G,A).$
It follows $$\theta_{G,A}^{m+n}(\alpha\cup\beta)-\theta_{G,A}^m(\alpha)\cup\beta - (-1)^m\alpha\cup\theta_{G,A}^n(\beta)\in B^{m+n}(G,A),$$
thus
$$\theta_{G,A}^{m+n}(\overline{\alpha}\cup\overline{\beta})=\theta_{G,A}^m(\overline{\alpha})\cup\overline{\beta}+(-1)^m\overline{\alpha}\cup\theta_{G,A}^n(\overline{\beta}).$$

\end{proof}
\begin{theo}\label{thm1}
  Let $\xymatrix{0\ar[r]&A\ar[r]^-{\iota}
&B\ar[r]^-{\pi}&A\ar[r] &0}$ be a short exact sequence in Situation \ref{sit*}. If there is a commutative diagram 
\begin{equation} \label{eq5}\xymatrix{ 0\ar[r]&A\ar[r]^-{\iota}&B\ar[r]^{\pi}&A \ar[r]&0\\
                                            0\ar[r]&C\ar[r]^{\iota_0}\ar[u]^-{\pi_0}&C\ar[r]^-{\pi_0}
                                            \ar[u]^{\pi_0'}&A\ar[r]\ar@{=}[u]&0 },
\end{equation}
  where $\xymatrix{0\ar[r]&C\ar[r]^-{\iota_0}
&C\ar[r]^-{\pi_0}&A\ar[r] &0}$ is a new  short exact sequence of $kG$-modules, then the triple $(\Hc^*(G,A),\cup,\theta_{G,A}^*)$ is a DG-algebra.
\end{theo}

\begin{proof}
By Proposition \ref{prop4thetagraded} we already know that $\theta_{G,A}^*$ is a graded derivation.   Denote by $(\pi)^n,$ $(\pi_0)^n, $ $(\pi_0')^n$ the homomorphisms induced in cohomology by the covariant cohomology functor $\Hc^*(G,-)$.  We apply the Long Exact Sequence Theorem to diagram (\ref{eq5}) obtaining the following commutative diagram

$$\xymatrix{ \Hc^n(G,B)\ar[r]^-{(\pi)^n}&\Hc^n(G,A)\ar[r]^-{\theta_{G,A}^n}&\Hc^{n+1}(G,A)\\
                        \Hc^n(G,C)\ar[r]^-{(\pi_0)^n}\ar[u]^-{(\pi_0')^n} & \Hc^n(G,A)\ar[r]^{\theta_{G,A,C}^n}\ar@{=}[u]&\Hc^{n+1}(G,C)
                                            \ar[u]^{(\pi_0)^{n+1}}},$$
    hence $\theta_{G,A}^n=(\pi_0)^{n+1}\circ \theta_{G,A,C}^n.$
   For $n\in\bbZ, n\geq 1$ it follows
    $$\theta_{G,A}^n\circ \theta_{G,A}^{n-1} = (\pi_0)^{n+1}\circ\theta_{G,A,C}^n\circ(\pi_0)^n\circ\theta_{G,A,C}^{n-1}=0,$$
   hence $\theta_{G,A}^*$ is a differential.
\end{proof}
An example of a short exact sequence in Situation \ref{sit*} and of a commutative diagram like above is given in the next remark. Here we recover the classical Bockstein map in group cohomology, see \cite[Definition 4.3.1]{BeII}.
\begin{rema}\label{remBock}
Let $p$ be a prime dividing the order of $G$. We consider the next diagram of trivial $\bbZ G$-modules (hence $k=\bbZ$)
\begin{equation*} \xymatrix{ 0\ar[r]&\bbZ/p\bbZ\ar[r]^-{\iota}&\bbZ/p^2\bbZ\ar[r]^{\pi}&\bbZ/p\bbZ\ar[r]&0\\
                                            0\ar[r]&\bbZ\ar[r]^{\iota_0}\ar[u]^-{\pi_0}&\bbZ\ar[r]^-{\pi_0}
                                            \ar[u]^{\pi_0'}&\bbZ/p\bbZ\ar[r]\ar@{=}[u]&0 },
\end{equation*}

where $$\iota(a+p\bbZ)=pa+p^2\bbZ, \pi(a+p^2\bbZ)=a+p\bbZ, \iota_0(a)=pa,\pi_0(a)=a+p\bbZ,\pi_0'(a)=a+p^2\bbZ,$$
for any $a\in\bbZ$. The isomorphism of $\bbZ G$-modules $r:\Ker \pi\rightarrow \bbZ/p\bbZ$ is given by $r(pa+p^2\bbZ)=a+p\bbZ,$ for any $a\in\bbZ$.
The section $s:\bbZ/p\bbZ\rightarrow\bbZ/p^2\bbZ$ is $s(a+p\bbZ)=\hat{a}+p^2\bbZ$ where $\hat{a}$ is the unique representative in $\{0,\ldots, p-1\}$ of $a+p\bbZ$, for any $a\in \bbZ$.
\end{rema}

\section{ Basic facts on Poisson $r$-algebras and BD algebras}\label{sec3}

Let $r,m,n$ be nonnegative integers and $K$ be a field. The purpose of this section is to collect basic definitions, from various references (see \cite[Section A.3.2]{CoGw}, \cite[Section 1]{CaFiLo}), of $P_r$ 
algebras (Poisson $r$-algebras)  and BD algebras. We recover in this way both definitions of BV algebras and BD algebras and analyze the similarity.
Firstly we recall the definition of an Poisson $r$-algebra, see \cite[Defintion 1.1]{CaFiLo}, but we adopt a strange looking index modification. In our definition the Lie bracket is of degree $1-r$ not $-r$. This is the indexing convention made in \cite[Definition 3.4]{CoGw}, but we do not start with a given cochain complex like Costello and Gvilliam. We start with a graded vector space and use the same name,  $P_r$ algebra, given in \cite{CoGw}.

\begin{defi}\label{def31Pr}(\cite[Definition 1.1]{CaFiLo}) A  $P_r$ algebra is a triple $(H^*,\cdot,[\cdot,\cdot])$, where $H^*=\bigoplus_{i\in\bbZ}H^i$ is a $\bbZ$-graded $K$-vector space, such that:
    \begin{itemize}
        \item[(1)] $(H^*,\cdot)$ is a graded commutative algebra;
        \item[(2)] $(H^*[r-1],[\cdot,\cdot])$ is a Lie algebra with a  Lie Bracket of degree $1-r$ (which is called $r$-Poisson bracket), hence $$[\cdot,\cdot]:H^n\times H^m\rightarrow H^{n+m+1-r};$$
        \item[(3)] Poisson identity is satisfied, that is $$[\alpha,\beta\cdot\gamma] = [\alpha,\beta]\cdot\gamma+(-1)^{(|\alpha|-r+1)|\beta|}\beta\cdot[\alpha,\gamma],$$
       for any $\alpha,\beta,\gamma\in H^*$ homogeneous elements.
        \end{itemize}
\end{defi}   

\begin{rema}\label{rem32}
\begin{itemize}
\item[(a)] A $P_0$ algebra $(H^*,\cdot, [\cdot, \cdot])$ is an algebraic structure on which we rely to  define what is called in this paper a BD algebra.
\item[(b)] A $P_1$ algebra  is  an ordinary (graded) Poisson algebra. In this case $H^0$ becomes and ordinary (non-graded) Poisson algebra.
\item[(c)] A $P_2$ algebra $(H^*,\cdot, [\cdot, \cdot])$ is precisely a Gerstenhaber algebra.
\end{itemize}
\end{rema}
In the following definition, using the indexing of \cite{CoGw}, we obtain $r-1$-Batalin Vilkovisky algebras of \cite[Definition 1.3]{CaFiLo}.

\begin{defi}\label{def33} (\cite[Definition 1.2]{CaFiLo})
A $BD_r$ algebra is a $P_r$ algebra  $(H^*,\cdot,[\cdot,\cdot])$  together with an operator $\Delta_{BD_r}:{H}^*\rightarrow H^{*+1-r} $ of degree $1-r$  such that $\Delta_{BD_r}\circ\Delta_{BD_r} = 0 $  and 

        $$ [\alpha,\beta] = (-1)^{|\alpha|}\Delta_{BD_r}(\alpha\cdot\beta)-(-1)^{|\alpha|}\Delta_{BD_r}(\alpha)\cdot\beta - \alpha\cdot\Delta_{BD_r}(\beta),$$
        for any $\alpha,\beta \in H^*$ homogeneous elements. 
\end{defi}
The advantage of the index convention in the above definition is that we can recover the BD algebra  (when $r=0$) and BV algebra concepts (when $r=2$).
\begin{rema}\label{rem34}
\begin{itemize}
\item[(a)] A $BD_0$ algebra is what we call in this paper a BD algebra, see \cite[p.288-289]{CoGw} and the discussion in \cite{BVBD}. In Introduction we have denoted the above operator $\Delta_{BD_0}$ by  $\Delta_{BD}$. The same notation is used for the rest of the paper.
\item[(b)] A $BD_2$ algebra is precisely a BV algebra or what is called a $1$-Batalin Vilkovisky algebra in \cite{CaFiLo}.
\end{itemize}
\end{rema}
Notice that in \cite[p.288-289]{CoGw} the authors define the  $P_0$  and $BD_0$ algebra concepts on a cochain complex $(H^*,d)$ and with the BD operator $\Delta_{BD_0}$ the same as $d$. For BV algebras, they have a BV operator $\Delta_{BD_2}:H^*\rightarrow H^{*-1}$ different than $d$, but the similar equation defining $\Delta_{BD_2}$ like in the case $\Delta_{BD_0}$. 
\section{Group cohomology and BD operators}\label{sec4}

In this section we put in scene all the setup needed to obtain  a method of constructing   BD algebras using the theory  about group cohomology as DG-algebras, obtained in Section \ref{sec2}. We shall use the following notations. If $M,N$ are $kG$-modules and $M'$ is a $kG$-submodule of $M$ we denote by
$\Hom_{kG}(M',N)$ the $k$-module of all $kG$-module homomorphisms from $M'$ to $N$. By $\Hom_{kG,M'}^M(M',N)$ we denote the $k$-submodule of $\Hom_{kG}(M',N)$ which contains all $kG$-module homomorphisms which are restrictions  to $M'$ of  $kG$-homomorphisms from $M$ to $N$. We will consider these $k$-modules $\Hom_{kG}(M',N), \Hom_{kG,M'}^M(M',N)$ as $kG$-modules with $G$ acting trivially.

\begin{subsec}\label{sit**}
Let $G$ be a finite group, $k$ be a commutative ring and  let $A,B$ be $k$-algebras with $G$ acting by automorphisms on $A$.
 Let $\pi:B\rightarrow A$ be a surjective homomorphism of $kG$-modules such that $\Ker \pi $ is an ideal of $B$ and $s:A\rightarrow B$ is a section of $\pi$.

Let $x\in G$. We assume that
\begin{itemize}
\item[(i)] either there is a  short exact sequence of $kG$-modules $$\xymatrix{0\ar[r]&A\ar[r]^-{\iota_x}
&B\ar[r]^-{\pi}&A\ar[r] &0}\qquad (1)_x$$ such that 
\begin{itemize}
\item[$\blacktriangleright$] $(1)_x$ is in Situation \ref{sit*}; explicitly, we denote by $r_x:\Ker \pi \rightarrow A$ the inverse of the isomorphism $\iota'_x:A\rightarrow \Ima \iota_x$ and $s:A\rightarrow B$ is the above section of $\pi$ satisfying the similar relations  (\ref{eq2setup*}), (\ref{eq3setup*}) and (\ref{eq4setup*}) of Situation \ref{sit*};
\item[$\blacktriangleright$] there is a short exact sequence $\xymatrix{0\ar[r]&C\ar[r]^-{\iota_{0,x}}
&C\ar[r]^-{\pi_0}&A\ar[r] &0}$of $kG$-modules giving a commutative diagram 
$$\xymatrix{ 0\ar[r]&A\ar[r]^-{\iota_x}&B\ar[r]^{\pi}&A \ar[r]&0\\
                                            0\ar[r]&C\ar[r]^{\iota_{0,x}}\ar[u]^-{\pi_0}&C\ar[r]^-{\pi_0}
                                            \ar[u]^{\pi_0'}&A\ar[r]\ar@{=}[u]&0 };$$
\end{itemize}
\item[(ii)] or $r_x:=0$ such that 

\item[(iii)] if $r:G\rightarrow \Hom_{kG}(\Ker \pi,A)/\Hom^B_{\Ker \pi, kG}(\Ker \pi, A)$ is the map $$x\mapsto r(x):=r_x+\Hom^B_{\Ker \pi, kG}(\Ker \pi, A)$$
then $r\in \Hc^1(G,\Hom_{kG}(\Ker \pi,A)/\Hom^B_{\Ker \pi, kG}(\Ker \pi, A))$

\end{itemize}
\end{subsec}
The above statement (iii) can be explicitly given by verifying that for any $y,z\in G$ there is $r'_{y,z}$ a homomorphism in $\Hom_{kG}(B,A)$ satisfying $$r_{yz}-r_y-r_z=r'_{y,z}|_{\Ker\pi}.$$
\begin{prop}\label{prop42} Assume we are in Situation \ref{sit**}. 
\begin{itemize}
\item[(i)] For any $x\in G$ the map $\theta_{G,A,x}^*:\Hc^*(G,A)\rightarrow \Hc^{*+1}(G,A)$ given by
$$\overline{\varphi}\mapsto \overline{\theta_{G,A,x}^*(\varphi)},\quad \theta_{G,A,x}^*(\varphi):=r_x\circ\partial_B(s\circ \varphi)$$
is a graded derivation between graded $k$-modules such that $(\Hc^*(G,A),\cup,\theta_{G,A,x}^*)$ is a  DG-algebra. Particularly, for some $x\in G$  the map $\theta_{G,A,x}^*$ may be trivial;
\item[(ii)] For any $x,y\in G$ we have $\theta_{G,A,xy}^*=\theta_{G,A,x}^*+\theta_{G,A,y}^*;$
\item[(iii)] For any $x,y\in G$ we have $\theta_{G,A,x}^{*+1}\circ\theta_{G,A,y}^*=0.$
\end{itemize}
\end{prop}
\begin{proof} Let $x\in G$ such that there is a short exact sequence $(1)_x$  in Situation \ref{sit*} . Then, we apply Theorem \ref{thm1} to obtain statement (i). If this is not the case, then by Situation \ref{sit**}, (ii) we obtain $\theta_{G,A,x}^*=0$.

For (ii) let $x,y\in G$ and $\varphi\in Z^*(G,A)$. We obtain 
\begin{align*}\theta_{G,A,xy}^*(\varphi)&=r_{xy}\circ \partial_B(s\circ \varphi)\\&=(r_x+r_y-r'_{x,y}|_{\Ker \pi})\circ \partial_B(s\circ \varphi)\\&=r_x\circ \partial_B(s\circ \varphi)+r_y\circ\partial_B(s\circ \varphi)+r'_{x,y}|_{\Ker \pi}\circ \partial_B(s\circ \varphi)\\&=\theta_{G,A,x}^*(\varphi)+\theta_{G,A,y}^*(\varphi)+\partial_A(r'_{x,y}\circ s\circ\varphi),
\end{align*}
hence the statement.

For (iii) a similar proof as in Theorem \ref{thm1} works. Let $n\in\bbZ, n\geq 1$. If both maps $\theta_{G,A,x}^{n},\theta_{G,A,y}^{n-1}$ are obtained as consequence of Situation \ref{sit**}, case (i), then we have two commutative diagrams
                                            
$$\xymatrix{ \Hc^n(G,B)\ar[r]^-{(\pi)^n}&\Hc^n(G,A)\ar[r]^-{\theta_{G,A,x}^n}&\Hc^{n+1}(G,A)\\
                        \Hc^n(G,C)\ar[r]^-{(\pi_0)^n}\ar[u]^-{(\pi_0')^n} & \Hc^n(G,A)\ar[r]^{\theta_{G,A,C,x}^n}\ar@{=}[u]&\Hc^{n+1}(G,C)
                                            \ar[u]^{(\pi_0)^{n+1}}},$$
$$\xymatrix{ \Hc^{n-1}(G,B)\ar[r]^-{(\pi)^n}&\Hc^{n-1}(G,A)\ar[r]^-{\theta_{G,A,y}^{n-1}}&\Hc^{n}(G,A)\\
                        \Hc^{n-1}(G,C)\ar[r]^-{(\pi_0)^{n-1}}\ar[u]^-{(\pi_0')^{n-1}} & \Hc^{n-1}(G,A)\ar[r]^{\theta_{G,A,C,y}^{n-1}}\ar@{=}[u]&\Hc^{n}(G,C)
                                            \ar[u]^{(\pi_0)^{n}}},$$
                                            hence
                                            
                                            $$\theta_{G,A,x}^{n}\circ\theta_{G,A,y}^{n-1}=(\pi_0)^{n+1}\circ \theta_{G,A,C,x}^{n}\circ (\pi_0)^n\circ \theta_{G,A,C,y}^{n-1}=0$$

\end{proof}
\begin{defi}\label{defBD43}
Assume we are in Situation \ref{sit**}. Let $K$ be a field such that $A$ has a $K$-algebra structure, with $G$ acting by automorphisms on $A$ and,  for any $x\in G$ the map $\theta_{G,A,x}^*$ becomes an operator of $K$-vector spaces. Let $x,y\in G$ and $\overline{\varphi}, \overline{\psi}\in \Hc^*(G,A)$ be homogeneous elements. On $KG\otimes \Hc^*(G,A)$ we define:
\begin{itemize}
\item[(1)] a multiplication "$\cdot$"
$$(x\otimes \overline{\varphi})\cdot (y\otimes \overline{\psi})=xy\otimes \overline{\varphi\cup \psi};$$
\item[(2)] a bilinear map $$[\cdot,\cdot]:(KG\otimes \Hc^m(G,A))\otimes (kG\otimes \Hc^n(G,A))\rightarrow KG\otimes \Hc^{m+n+1}(G,A),$$
$$[x\otimes \overline{\varphi}, y\otimes \overline{\psi}]=xy\otimes\left( (-1)^{m}\theta_{G,A,y}^m(\overline{\varphi})\cup\overline{\psi}+\overline{\varphi}\cup \theta_{G,A,x}^n(\overline{\psi})\right);$$
\item[(3)] an operator
$$\Delta_{BD}:KG\otimes \Hc^*(G,A)\rightarrow KG\otimes \Hc^{*+1}(G,A),\quad \Delta_{BD}(x\otimes \overline{\varphi})=x\otimes \theta_{G,A,x}^*(\overline{\varphi}).$$
\end{itemize}
\end{defi}
\begin{rema}\label{rem44} Let $A$ be a $k$-algebra on which $G$ acts by $k$-automorphisms. If $k$ is a commutative unital ring let $K=k/m$ be  the residue field  with the maximal ideal $m$ acting trivially on $A$. Then $A$ has a $K$-algebra structure, with $G$ acting by $K$-automorphisms on $A$. Moreover $\Hc^*(G,A)$ inherits a structure of graded $K$-algebra and, for any $x\in G$ the map $\theta_{G,A,x}^*$ is an operator of $K$-vector spaces.
\end{rema}
\begin{theo}\label{thm2} With the assumptions of Definition \ref{defBD43}, if $G$ is abelian, then $(KG\otimes \Hc^*(G,A),\cdot, [\cdot,\cdot],\Delta_{BD})$ is a BD algebra.
\end{theo}
\begin{proof}
Let $x,y,z\in G$ and $\varphi\in Z^m(G,A), \psi\in Z^n(G,A), \omega\in Z^t(G,A),$ hence $m=|\overline{\varphi}|, n=|\overline{\psi}|, t=|\overline{\omega}|$.

By Definition \ref{def33} and Remark \ref{rem34} (a) we need to show that $(KG\otimes\Hc^*(G,A),\cdot, [\cdot,\cdot])$ is a $P_0$ algebra and $\Delta_{BD}$ is a differential operator of degree $+1$ such that 
$$ [x\otimes \overline{\varphi},y\otimes\overline{\psi}] = (-1)^{m}\Delta_{BD}((x\otimes \overline{\varphi})\cdot(y\otimes \overline{\psi}))-(-1)^{m}\Delta_{BD}(x\otimes\overline{\varphi})\cdot(y\otimes \overline{\psi}) - (x\otimes \overline{\varphi})\cdot\Delta_{BD}(y\otimes\overline{\psi}).$$

In order to verify that $(KG\otimes \Hc^*(G,A),\cdot, [\cdot,\cdot])$ is a $P_0$ algebra it is easy to show that $(KG\otimes \Hc^*(G,A),\cdot)$ is a graded commutative algebra and $[x\otimes \overline{\varphi}, y\otimes\overline{\psi}]=(-1)^{(m-1)(n-1)}[y\otimes\overline{\psi},x\otimes\overline{\varphi}]$.

First we check the graded Jacob identity of $(KG\otimes \Hc^*(G,A))[-1]$ and the Poisson  identity.  Let $$A:=(-1)^{(m-1)(t-1)}[[x\otimes \overline{\varphi}, y\otimes\overline{\psi}],z\otimes \overline{\omega}], \ \ B:=(-1)^{(n-1)(m-1)}[[ y\otimes\overline{\psi},z\otimes \overline{\omega}],x\otimes \overline{\varphi} ], $$  $$C:=(-1)^{(t-1)(n-1)}[[z\otimes \overline{\omega},x\otimes \overline{\varphi} ], y\otimes\overline{\psi}].$$

We have
\begin{align*}
&A=(-1)^{(m-1)(t-1)}[xy\otimes( (-1)^m\theta_{G,A,y}^m(\overline{\varphi})\cup\overline{\psi}+\overline{\varphi}\cup\theta_{G,A,x}^n(\overline{\psi})),z\otimes \overline{\omega}]\\
&=(-1)^{(m-1)(t-1)}\left([xy\otimes( (-1)^m\theta_{G,A,y}^m(\overline{\varphi})\cup\overline{\psi}),z\otimes\overline{\omega}]+[xy\otimes(\overline{\varphi}\cup\theta_{G,A,x}^n(\overline{\psi})),z\otimes \overline{\omega}]\right)\\
&=(-1)^{(m-1)( t-1)} \left(xyz\otimes( (-1)^{m+n+1}\theta_{G,A,z}^{m+n+1}((-1)^m\theta_{G,A,y}^m(\overline{\varphi})\cup\overline{\psi})\cup \overline{\omega}+(-1)^m\theta_{G,A,y}(\overline{\varphi})\cup\overline{\psi}\cup\right.\\& \left.\theta_{G,A,xy}^t(\overline{\omega}) +(-1)^{m+n+1}\theta_{G,A,z}^{m+n+1}(\overline{\varphi}\cup\theta_{G,A,x}^n(\overline{\psi}))\cup\overline{\omega}+\overline{\varphi}
\cup\theta_{G,A,x}^n(\overline{\psi})\cup\theta_{G,A,xy}^t(\overline{\omega}))\right) .
\end{align*}

Next we apply Proposition \ref{prop42} (i), (iii) to obtain
\begin{align*}
A&=(-1)^{(m-1)( t-1)} \left(xyz\otimes( (-1)^{m+n}\theta_{G,A,y}^m(\overline{\varphi})\cup\theta_{G,A,z}^n(\overline{\psi})\cup \overline{\omega}+(-1)^m\theta_{G,A,y}(\overline{\varphi})\cup\overline{\psi}\cup\theta_{G,A,xy}^t(\overline{\omega}) \right.\\& \left.+(-1)^{m+n+1}\theta_{G,A,z}^{m}(\overline{\varphi})\cup\theta_{G,A,x}^n(\overline{\psi})\cup\overline{\omega}+\overline{\varphi}
\cup\theta_{G,A,x}^n(\overline{\psi})\cup\theta_{G,A,xy}^t(\overline{\omega}))\right).
\end{align*}
Similarly
\begin{align*}
B&=(-1)^{(n-1)( m-1)} \left(xyz\otimes( (-1)^{n+t}\theta_{G,A,z}^n(\overline{\psi})\cup\theta_{G,A,x}^t(\overline{\omega})\cup \overline{\varphi}+(-1)^n\theta_{G,A,z}^n(\overline{\psi})\cup\overline{\omega}\cup\theta_{G,A,yz}^m(\overline{\varphi}) \right.\\& \left.+(-1)^{n+t+1}\theta_{G,A,x}^{n}(\overline{\psi})\cup\theta_{G,A,y}^t(\overline{\omega})\cup\overline{\varphi}+\overline{\psi}
\cup\theta_{G,A,y}^t(\overline{\omega})\cup\theta_{G,A,yz}^m(\overline{\varphi}))\right).
\end{align*}
\begin{align*}
C&=(-1)^{(t-1)(n-1)} \left(xyz\otimes( (-1)^{t+m}\theta_{G,A,x}^t(\overline{\omega})\cup\theta_{G,A,y}^m(\overline{\varphi})\cup \overline{\psi}+(-1)^t\theta_{G,A,x}^t(\overline{\omega})\cup\overline{\varphi}\cup\theta_{G,A,zx}^n(\overline{\psi}) \right.\\& \left.+(-1)^{t+m+1}\theta_{G,A,y}^{t}(\overline{\omega})\cup\theta_{G,A,z}^m(\overline{\varphi})\cup\overline{\psi}+\overline{\omega}
\cup\theta_{G,A,z}^m(\overline{\varphi})\cup\theta_{G,A,zx}^n(\overline{\psi}))\right).
\end{align*}
Now taking into account the signs, the fact that $"\cup"$ is graded commutative and Proposition 4.2 (ii) we obtain
$$A+B+C = 0.$$
For Poisson identity we obtain
\begin{align*}
    [x\otimes\overline{\varphi},(y\otimes\overline{\psi})\cdot(z\otimes\overline{\omega})] &= [x\otimes\overline{\varphi},yz\otimes\overline{\psi}\cup\overline{\omega})]\\
    &= xyz \otimes\left((-1)^m\theta_{G,A,yz}^m(\overline{\varphi})\cup\overline{\psi}\cup\overline{\omega}+\overline{\varphi}\cup\theta_{G,A,x}^{n+t}(\overline{\psi}\cup\overline{\omega}) \right)\\
    &=xyz\otimes\left((-1)^m\theta_{G,A,y}^m(\overline{\varphi})\cup\overline{\psi}\cup\overline{\omega}+(-1)^m\theta_{G,A,z}^m(\overline{\varphi})\cup\overline{\psi}\cup\overline{\omega} \right.\\
    &+\left. \overline{\varphi}\cup\theta_{G,A,x}^{n}(\overline{\psi})\cup\overline{\omega}+(-1)^n\overline{\varphi}\cup\overline{\psi}\cup\theta_{G,A,x}^t(\overline{\omega}) \right) ,\numberthis \label{eq6}
\end{align*} where for the third equality we used Proposition \ref{prop42} (i), (ii).
\begin{align*}
    [& x\otimes\overline{\varphi}, y\otimes \overline{\psi}]\cdot(z\otimes\overline{\omega})+(-1)^{(m-1)n}(y\otimes\overline{\psi})\cdot[x\otimes\overline{\varphi},z\otimes\overline{\omega}]\\
    =& (xy\otimes((-1)^m\theta_{G,A,y}^m(\overline{\varphi})\cup\overline{\psi}+\overline{\varphi}\cup\theta_{G,A,x}^n(\overline{\psi})))\cdot(z\otimes\overline{\omega})+ (-1)^{(m-1)n}(y\otimes\overline{\psi})\cdot(xz\otimes((-1)^m\theta_{G,A,z}^m(\overline{\varphi})\cup\overline{\omega}\\
    &+\overline{\varphi}\cup\theta_{G,A,x}^t(\overline{\omega})))  \\
    =&  xyz\otimes\left( (-1)^m\theta_{G,A,y}^m(\overline{\varphi})\cup\overline{\psi}\cup\overline{\omega}+\overline{\varphi}\cup\theta_{G,A,x}^n(\overline{\psi})\cup\overline{\omega}+(-1)^{(m-1)n+m}\overline{\psi}\cup\theta_{G,A,z}^m(\overline{\varphi})\cup\overline{\omega}+\right.\\
    &\left.+ (-1)^{(m-1)n}\overline{\psi}\cup\overline{\varphi}\cup\theta_{G,A,x}^t(\overline{\omega}) \right) \\
    =&  xyz\otimes\left( (-1)^m\theta_{G,A,y}^m(\overline{\varphi})\cup\overline{\psi}\cup\overline{\omega}+\overline{\varphi}\cup\theta_{G,A,x}^n(\overline{\psi})\cup\overline{\omega}+(-1)^{(m-1)n+m+(m+1)n}\theta_{G,A,z}^m(\overline{\varphi})\cup\overline{\psi}\cup\overline{\omega}+\right.\\
    &\left.+ (-1)^{(m-1)n+mn}\overline{\varphi}\cup\overline{\psi}\cup\theta_{G,A,x}^t(\overline{\omega}) \right) \numberthis \label{eq7}\\
\end{align*}
Since $$(-1)^{(m-1)n+m+(m+1)n} = (-1)^{2mn+m} = (-1)^{m}$$ and $$(-1)^{(m-1)n+mn} = (-1)^{2mn - n} = (-1)^{-n} = (-1)^n$$ the expressions (\ref{eq6}) and (\ref{eq7}) are the same.

Finally it is easy to verify $\Delta_{BD}\circ\Delta_{BD} = 0$. We are left to prove the last identity of Definition \ref{def33}. On one hand
\begin{equation}\label{eq8}
    [x\otimes\overline{\varphi}, y\otimes\overline{\psi}] = xy \otimes ((-1)^m\theta_{G,A,y}^m(\overline{\varphi})\cup\overline{\psi} + \overline{\varphi}\cup\theta_{G,A,x}^n(\overline{\psi})) 
    \end{equation}
   On the other hand we have
\begin{align*}
     (-1)^m&\Delta_{BD}(xy\otimes(\overline{\varphi}\cup\overline{\psi})) - (-1)^m\Delta_{BD}(x\otimes\overline{\varphi})\cdot (y\otimes\overline{\psi}) - (x\otimes\overline{\varphi})\cdot\Delta_{BD}(y\otimes\overline{\psi}) \\
    =& (-1)^m xy\otimes( \theta_{G,A,x}^{m+n}(\overline{\varphi}\cup\overline{\psi}) + \theta_{G,A,y}^{m+n}(\overline{\varphi}\cup\overline{\psi}) ) - (-1)^m xy\otimes(\theta_{G,A,x}^m(\overline{\varphi})\cup\overline{\psi}) - xy \otimes(\overline{\varphi}\cup\theta_{G,A,y}^n(\overline{\psi}))\\
    =& xy\otimes\left((-1)^m  \theta_{G,A,x}^{m}(\overline{\varphi})\cup\overline{\psi} + \overline{\varphi}\cup \theta_{G,A,x}^{n}(\overline{\psi}) +  (-1)^m \theta_{G,A,y}^m(\overline{\varphi})\cup\overline{\psi} +\overline{\varphi}\cup\theta_{G,A,y}^n(\overline{\psi}) - \right.\\
    &\left.- (-1)^m\theta_{G,A,x}^m(\overline{\varphi})\cup\overline{\psi} - \overline{\varphi}\cup\theta_{G,A,y}^n(\overline{\psi})\right)\\
    =& xy \otimes((-1)^m\theta_{G,A,y}^m(\overline{\varphi})\cup\overline{\psi}+\overline{\varphi}\cup\theta_{G,A,x}^n(\overline{\psi}))  \numberthis \label{eq9}
\end{align*}
Comparing (\ref{eq8}) and (\ref{eq9}) we are done. 
\end{proof}

\section{Short exact sequences inducing BD operators}\label{sec5}
In this section we present examples of short exact sequences in Situation \ref{sit**}. Let $G=C_3=<x>$ be the cyclic group of order $3$ and let $k=\bbZ$. Let $\pi:\bbZ/9\bbZ\rightarrow\bbZ/3\bbZ$ be the map $\pi(a+9\bbZ)=a+3\bbZ,$ for any $a\in\bbZ$. It is clear that $\pi$ is a surjective homomorphism of rings and we choose $s:\bbZ/3\bbZ\rightarrow\bbZ/9\bbZ$ a section of $\pi$, defined by $s(a+3\bbZ)=\hat{a}+9\bbZ,$ for any $a\in\bbZ$; here $\hat{a}$ is the unique element of $\{0,1,2\}$ such that $a-\hat{a}\in 3\bbZ$. The ideal $\Ker \pi$ is $\{3a+9\bbZ| a\in\bbZ\}$ and $\bbZ/3\bbZ, \bbZ/9\bbZ$ are $\bbZ$-algebras with $C_3$ acting trivially on $\bbZ/3\bbZ$. 

We define the following  short exact sequences of trivial $\bbZ G$-modules.
\subsection{A short exact sequence indexed by $x\in C_3$.}\label{subsec51}

Let $$\xymatrix{0\ar[r]&\bbZ/3\bbZ\ar[r]^-{\iota_x}
&\bbZ/9\bbZ\ar[r]^-{\pi}&\bbZ/3\bbZ\ar[r] &0}$$ 
be the short exact sequence given by $\iota_x(a+3\bbZ)=3a+9\bbZ$, for any $a\in\bbZ$. In this case the induced isomorphism
$\iota'_x:\bbZ/3\bbZ\rightarrow \Ker \pi$ of trivial $\bbZ G$-modules, defined by $\iota'_x(a+3\bbZ)=3a+9\bbZ$, has its inverse
$r_x:\Ker \pi \rightarrow \bbZ/3\bbZ$ given by $r_x(3a+9\bbZ)=a+3\bbZ,$ for any $a\in\bbZ$. By Remark \ref{rem24} we only have to verify (\ref{eq2setup*}) and (\ref{eq3setup*}), which are easy.

There is also a commutative diagram
\begin{equation*} \xymatrix{ 0\ar[r]&\bbZ/3\bbZ\ar[r]^-{\iota_x}&\bbZ/9\bbZ\ar[r]^{\pi}&\bbZ/3\bbZ\ar[r]&0\\
                                            0\ar[r]&\bbZ\ar[r]^{\iota_{0,x}}\ar[u]^-{\pi_0}&\bbZ\ar[r]^-{\pi_0}
                                            \ar[u]^{\pi_0'}&\bbZ/3\bbZ\ar[r]\ar@{=}[u]&0 },
\end{equation*}
where $\iota_{0,x}(a)=3a, \pi_0(a)=a+3\bbZ, \pi_0'(a)=a+9\bbZ,$ for any $a\in \bbZ$.

\subsection{A short exact sequence indexed by $x^2\in C_3$.}\label{subsec52}
Let $$\xymatrix{0\ar[r]&\bbZ/3\bbZ\ar[r]^-{\iota_{x^2}}
&\bbZ/9\bbZ\ar[r]^-{\pi}&\bbZ/3\bbZ\ar[r] &0}$$ 
be the short exact sequence given by $\iota_{x^2}(a+3\bbZ)=-3a+9\bbZ$, for any $a\in\bbZ$. In this case the induced isomorphism
$\iota'_{x^2}:\bbZ/3\bbZ\rightarrow \Ker \pi$ of trivial $\bbZ G$-modules, defined by $\iota'_{x^2}(a+3\bbZ)=-3a+9\bbZ$, has its inverse
$r_{x^2}:\Ker \pi \rightarrow \bbZ/3\bbZ$ given by $r_{x^2}(3a+9\bbZ)=-a+3\bbZ,$ for any $a\in\bbZ$. By Remark \ref{rem24} we only have to verify (\ref{eq2setup*}) and (\ref{eq3setup*}), which are left for the reader.

There is a commutative diagram with the same $\pi_0$ and $\pi_0'$ as in  \ref{subsec51}
\begin{equation*} \xymatrix{ 0\ar[r]&\bbZ/3\bbZ\ar[r]^-{\iota_{x^2}}&\bbZ/9\bbZ\ar[r]^{\pi}&\bbZ/3\bbZ\ar[r]&0\\
                                            0\ar[r]&\bbZ\ar[r]^{\iota_{0,x^2}}\ar[u]^-{\pi_0}&\bbZ\ar[r]^-{\pi_0}
                                            \ar[u]^{\pi_0'}&\bbZ/3\bbZ\ar[r]\ar@{=}[u]&0 },
\end{equation*}
where $\iota_{0,x^2}(a)=-3a, $ for any $a\in \bbZ$.

\subsection{Statement (iii) of Situation \ref{sit**} for the above short exact sequences.}\label{subsec53} We set $r_1:=0$ and it is an easy exercise to verify $r_{yz}-r_y-r_z=0$ for any $y,z\in C_3$. By \ref{subsec51}, \ref{subsec52}, \ref{subsec53} and Remark \ref{rem44} (with $k=\bbZ$ and $m=3\bbZ$) it follows that we are in Situation \ref{sit**} and, by Theorem \ref{thm2}, we obtain that $(\mathbb{F}_3C_3\otimes \Hc^*(C_3,\mathbb{F}_3),\cdot, [\cdot,\cdot], \Delta_{BD})$ is a BD algebra.
\begin{rema} Assume $C_3$ acts nontrivially on some finite group $H$. Then $\bbZ H, \mathbb{F}_3 H, (\bbZ/9\bbZ) H$ are rings with $G$ acting by automorphisms on $\mathbb{F}_3 H$. The above short exact sequences can be adapted to obtain other short exact sequences (with $G$ acting nontrivially) which remain in Situation \ref{sit**}. 
\end{rema}
\subsection{Explicitations of BD operators.}
It is well known that  $\Hc^*(C_3,\mathbb{F}_3)$ is $\mathbb{F}_3$ for any $*$ and, $\Hc^*(C_3,\bbZ)$ is $\mathbb{F}_3$ for $*$ even  and zero for $*$ odd. We may notice in \ref{subsec51} that $\theta_{G,\bbZ/3\bbZ,x}^*$ is precisely the ordinary Bockstein map for $p=3$, denoted
$\beta:\Hc^*(C_3,\mathbb{F}_3)\rightarrow \Hc^{*+1}(C_3,\mathbb{F}_3)$ in \cite[Definition 4.3.1]{BeII}. By the Long Exact Sequence Theorem applied to the commutative diagram corresponding to $(1)_x$ it follows $$\theta_{G,\bbZ/3\bbZ,x}^*=\pi_0^*\circ \theta_{G,\bbZ/3\bbZ,\bbZ,x}^*.$$  Since $\theta_{G,\bbZ/3\bbZ,\bbZ,x}^*$ is an isomorphism for $*$ odd and zero for $*$ even then $\theta_{G,\bbZ/3\bbZ,x}^*$ is an isomorphism for $*$ odd and zero for $*$ even. The same phenomenon holds for $\theta_{G,\bbZ/3\bbZ,x^2}=-\theta_{G,\bbZ/3\bbZ,x}$. 

 The algebra $\Hc^*(C_3,\mathbb{F}_3)$ contains a polynomial subalgebra $\mathbb{F}_3[\overline{\varphi}]$, with $\overline{\varphi}$ of degree $2$ such that $\Hc^*(C_3,\mathbb{F}_3)$ is generated as a module over
 $\mathbb{F}_3[\overline{\varphi}]$ by $1$ and an element $\overline{\psi}$ of degree one. Let $g\in C_3\setminus \{1\}$ and $n$ be any nonnegative integer. Using that $\theta_{C_3,\bbZ/3\bbZ,g}^*$ is a derivation  by Proposition \ref{prop42} (i), we obtain $$\theta_{C_3,\bbZ/3\bbZ,g}^{2n}(\overline{\varphi}^n)=0,\quad\theta_{C_3,\bbZ/3\bbZ,g}^{2n+1}(\overline{\varphi}^n\cup \overline{\psi})=\overline{\varphi}^n\cup \theta_{C_3,\bbZ/3\bbZ,g}^1(\overline{\psi}),$$ hence
\begin{align*}
&\Delta_{BD}(g\otimes \overline{\varphi}^n)=0,\\
&\Delta_{BD}(g\otimes(\overline{\varphi}^n\cup\overline{\psi}))=g\otimes (\overline{\varphi}^n\cup \theta_{C_3,\bbZ/3\bbZ,g}^1(\overline{\psi})).
\end{align*}

\vskip1cm

\textbf{Acknowledgements.} The author  would like to thank the referee for his/her careful reading and valuable comments.


\begin{thebibliography}{99}
\bibitem{AnDu} A. Angel, D. Duarte, The BV-algebra structure of the Hochschild cohomology of the group ring of
cyclic groups of prime order,in: \textit{Geometric, algebraic and topological methods for quantum field theory},
World Sci. Publ., Hackensack, NJ, 2017, pp. 353–-372.
\bibitem{BeII} D. J. Benson, \textit{Representations and cohomology II: Cohomology of groups and modules}. (Cambridge University Press, Cambridge, 1991).
\bibitem{BeKeLi}  D. J. Benson, R. Kessar, M. Linckelmann, "On the BV structure of the Hochschild cohomology of finite group algebras", arXiv:2005.01694 [math.RT].
\bibitem{CaFiLo} A. S. Cattaneo, D. Fiorenza, R. Longoni, Graded Poisson algebras, in: Francoise, J P; Naber, G L; Tsun, T S., \textit{Encyclopedia of Mathematical Physics}, Amsterdam, 2006, pp. 560--567.
\bibitem{CiSo} C. Cibils, A. Solotar, "Hochschild cohomology algebra of abelian groups", \emph{Arch. Math.} \textbf{68} (1997), 17--21.
\bibitem{CoGw} K. Costello, O. Gwilliam, \textit{Factorization algebras in quantum field theory, Volume 1}. (Cambridge University Press, 2017).
\bibitem{Ge} M. Gerstenhaber, "The cohomology structure of an associative ring", \textit{Ann. Math.} {\bf(2) 78} (1963), 267--288.
\bibitem{Get} E. Getzler,  "Batalin-Vilkovisky algebra and two dimensional topological fields theory", \textit{ Comm. Math. Phys.} {\bf159} (1994), 265--285.
\bibitem{LiZh} Y. Liu,  G. Zhou, "The Batalin-Vilkovisky structure over the Hochschild
cohomology ring of a group algebra", \textit{J. Noncomm. Geom.} {\bf10} (2016), 811--858.
\bibitem{BVBD} nLab authors. Relation between {{B}}{{V}} and {{B}}{{D}}. http://ncatlab.org/nlab/revision/relation\%20between\\ \%20BV\%20and\%20BD/4 (accessed April 2021).
\bibitem{Tr} T. Tradler, "The Batalin-Vilkovisky algebra on Hochschild cohomology induced by infinity inner products", \textit{Ann. Inst. Fourier} {\bf58 no.7} (2008), 2351--2379.

\end{thebibliography}
\end{document}